\documentclass[preprint,12pt]{elsarticle}

\usepackage{amsmath,amssymb,amsthm}
\usepackage{enumitem}
\usepackage{geometry}
\usepackage{hyperref}
\usepackage[utf8]{inputenc}
\usepackage{mathtools}
\newtheorem{theorem}{Theorem}[section]

\newtheorem{proposition}[theorem]{Proposition}
\newtheorem{corollary}[theorem]{Corollary}
\newtheorem{definition}[theorem]{Definition}
\newtheorem{remark}[theorem]{Remark}
\newtheorem{example}[theorem]{Example}

\begin{document}

\begin{frontmatter}

\title{Local Antisymmetric Connectedness in Quasi-Uniform and Quasi-Modular Spaces}

\author{Philani Rodney Majozi}
\ead{Philani.Majozi@nwu.ac.za}
\address{Department of Mathematics, Pure and Applied Analytics,
North-West University, Mahikeng, South Africa}

\begin{abstract}
Directional notions in topology and analysis naturally lead to
nonsymmetric structures such as quasi-metrics, quasi-uniformities, and
modular spaces. In these settings, classical notions of connectedness
and completion based on symmetric uniformities are often inadequate.
In this paper, we study \emph{antisymmetric connectedness} and
\emph{local antisymmetric connectedness} within the setting of
quasi-uniform and quasi-modular pseudometric spaces. We associate to
each quasi-modular pseudometric family compatible forward and backward
modular topologies and quasi-uniformities, yielding a canonical
bitopological structure. Using this setting, we establish
characterization and stability results for local antisymmetric
connectedness, including invariance under subspaces, uniformly
continuous mappings, and bicompletion. We further relate these notions
to Smyth completeness and Yoneda-type completions and show how
precompactness combined with asymmetric completeness yields compactness
in the join topology. Applications to asymmetric normed and modular
spaces illustrate the theory.
\end{abstract}

\begin{keyword}
Asymmetric normed spaces \sep quasi-metric spaces \sep quasi-uniform spaces \sep
quasi-modular spaces \sep antisymmetric connectedness \sep local antisymmetric connectedness \sep
Smyth completion \sep Yoneda completion
\\[2pt]
\end{keyword}

\end{frontmatter}
\section{Introduction}\label{sec:intro}

Many notions in analysis and topology are fundamentally directional:
costs need not be symmetric, approximation may be one-sided, and
convergence can depend on a prescribed orientation of nearness.
Such notions are naturally modelled by quasi-metrics and
quasi-uniform structures, where the absence of symmetry gives rise to
two conjugate notions of proximity.
This leads canonically to a pair of associated topologies and hence to a
bitopological environment
\cite{FletcherLindgren,Isbell1964,Kunzi2001,MarinRomaguera1998}.
The join of these topologies yields a symmetric envelope that is useful
both for comparison with the classical theory and for isolating features
that are genuinely directional from those that persist under
symmetrization \cite{Kelley,Willard}.

A central theme in nonsymmetric topology is that completion is no longer
governed by a single universal notion.
Beyond bicompletion via the supremum uniformity, one encounters
filter-pair based notions such as $D$-completeness and its variants, as
well as unilateral notions built from left/right $K$-Cauchy behavior and
Smyth-type completeness conditions
\cite{Doitchinov1991,FletcherLindgren,Kunzi2001}.
In parallel, the categorical point of view initiated by Lawvere and
developed in enriched settings provides a unifying perspective:
quasi-metric structures may be viewed as enriched categories, and
Yoneda-type constructions yield canonical completion procedures that
connect naturally with Smyth completions
\cite{Kelly1982,KunziSchellekens2002,Lawvere1973}.
Further motivation arises from computational and quantitative models,
where formal balls and related constructions encode completion and
compactness properties in a manner compatible with directed semantics
\cite{RomagueraTiradoValero2012,RomagueraValero2009,Rutten1998}; see also
\cite{FlaggSuenderhauf2002,Sunderhauf1998,Vickers1997}.

\medskip

Connectedness notions in nonsymmetric settings have so far received
considerably less attention than their metric and uniform counterparts.
In particular, recent work by Javanshir and Yıldız
\cite{JavanshirYildiz2023} investigates antisymmetric connectedness and
local antisymmetric connectedness in the restricted context of
asymmetrically normed real vector spaces, using path-based and
component-theoretic methods induced by norm-generated $T_0$-quasi-metrics.
While this approach provides valuable insight at the level of asymmetric
norms, it does not address connectedness from a quasi-uniform or
bitopological point of view, nor does it interact with completion theory,
bicompletion invariance, or modular structures.

\medskip

The purpose of this paper is to develop a substantially more general
setting in which \emph{directional connectedness} and
\emph{directional completion} notions can be formulated and studied in a
way that interfaces cleanly with the established quasi-uniform,
bitopological, and modular theories.
On the analytic side, we work with quasi-modular pseudometric families,
drawing on the classical modular space tradition and its modern
extensions to modular metric spaces
\cite{Chistyakov2015,MushaandjaOlela2025,Musielak1983}.
Building on our quasi-modular pseudometric approach \cite{Majozi2025},
we associate to a given modular family forward and backward modular
topologies, together with their induced quasi-uniformities and the
resulting bitopological structure.

Within this directional environment we introduce the notions of
\emph{antisymmetric connectedness} and \emph{local antisymmetric
connectedness} via bitopological separation principles rather than
path-based constructions.
This allows the theory to apply in settings where paths are unavailable
or inadequate, and makes the resulting notions stable under fundamental
operations such as uniformly continuous mappings and bicompletion.

Our contributions are primarily conceptual and structural.
First, we provide a systematic transition from quasi-modular pseudometric
families to compatible quasi-uniform and bitopological structures,
clarifying the role of conjugation and the join topology in the modular
context.
Second, we introduce directional connectedness notions adapted to this
setting and analyze their fundamental properties, including component
decompositions, local characterizations, and stability under subspaces
and images.
Third, we develop completion constructions suited to the quasi-modular
setting and aligned with the asymmetric completion paradigm; in
particular, we identify conditions under which precompactness combined
with Smyth-type completeness yields compactness, and we relate these
results to Yoneda-style completions and the formal-ball point of view
\cite{Kunzi2001,KunziSchellekens2002,Rutten1998}.
Finally, we indicate how these structural results interact with
asymmetric functional-analytic contexts, where nonsymmetric norms and
one-sided duality notions provide additional motivation
\cite{Cobzas2013,CobzasSurvey}.

The paper is organized as follows.
Section~\ref{sec:preliminaries} collects the quasi-metric,
quasi-uniform, and related preliminaries needed later.
In Section~\ref{sec:quasi-modular} we introduce quasi-modular
pseudometric families and develop their induced topologies and
quasi-uniformities.
Section~\ref{sec:completion} discusses completion mechanisms relevant to
the modular setting, with emphasis on Smyth-type completeness and the
Yoneda perspective.
Directional connectedness and its local counterpart are developed in
Sections~\ref{sec:sym-antisym-connectedness} and
\ref{sec:local-antisym-connectedness}.
Functional-analytic consequences and illustrative examples are
presented in Sections~\ref{sec:functional-analytic-consequences} and
\ref{sec:examples-applications}, followed by concluding remarks and open
directions in Section~\ref{sec:conclusion}.

\section{Preliminaries}\label{sec:preliminaries}

In this section we collect background material that will be used
throughout the paper. In line with the directional philosophy explained
in the Introduction (Section~\ref{sec:intro}), nonsymmetric proximity
notions naturally produce two conjugate topologies and hence a canonical
bitopological environment. Our basic references for quasi-uniform and
quasi-metric structures are \cite{FletcherLindgren,Isbell1964,Kunzi2001}.

\subsection{Quasi-Metric Spaces}\label{ssec:qm}

Quasi-metrics retain the triangle inequality but drop symmetry, thereby
encoding \emph{directional} proximity. This immediately produces two
topologies: the \emph{forward} topology $\tau_d$ generated by forward
balls
\[
B^+_\varepsilon(x):=\{y\in X: d(x,y)<\varepsilon\},
\]
and the \emph{backward} topology $\tau_{d^{-1}}$ generated by
\[
B^-_\varepsilon(x):=\{y\in X: d(y,x)<\varepsilon\},
\qquad\text{where } d^{-1}(x,y):=d(y,x).
\]
Consequently, the appropriate notion of Cauchy behavior becomes
direction-dependent. The left $K$-Cauchy condition captures
\emph{forward stabilization} of the net/sequence (tails become uniformly
close when read forward), while its associated limit notion is a
\emph{forward limit}. These unilateral notions underpin asymmetric
completion procedures and connect naturally with Smyth/Yoneda-type
constructions \cite{Kunzi2001,KunziSchellekens2002} and with Lawvere’s
enriched categorical point of view  \cite{Kelly1982,Lawvere1973}.

\begin{definition}\label{def:leftKCauchy}
Let $(X,d)$ be a quasi-pseudo-metric space (in particular, a quasi-metric
space if $(X,\tau_d)$ is $T_0$). A net $(x_\alpha)_{\alpha\in\Lambda}$ in
$X$ is called \emph{left $K$-Cauchy} if for every $\varepsilon>0$ there
exists $\alpha_\varepsilon\in\Lambda$ such that for all
$\alpha_2\ge \alpha_1\ge \alpha_\varepsilon$ we have
\[
d(x_{\alpha_1},x_{\alpha_2})<\varepsilon.
\]
A sequence $(x_n)_{n\in\mathbb{N}}$ is left $K$-Cauchy if the same
condition holds with indices $m\ge n\ge n_\varepsilon$:
\[
d(x_n,x_m)<\varepsilon\quad\text{for all }m\ge n\ge n_\varepsilon.
\]
Equivalently, if $\mathcal{U}_d$ denotes the quasi-uniformity induced by
$d$ (with basic entourages
$U_\varepsilon:=\{(x,y):d(x,y)<\varepsilon\}$), then $(x_\alpha)$ is left
$K$-Cauchy iff
\[
\forall U\in\mathcal{U}_d\ \exists \alpha_U\ \forall
\alpha_2\ge\alpha_1\ge\alpha_U:\ (x_{\alpha_1},x_{\alpha_2})\in U.
\]
\end{definition}

\begin{definition}
\label{def:forwardlimit}
Let $(X,d)$ be a quasi-pseudo-metric space and let
$(x_\alpha)_{\alpha\in\Lambda}$ be a left $K$-Cauchy net. A point $x\in
X$ is called a \emph{forward limit} of $(x_\alpha)$ if
\[
\lim_{\alpha} d(x_\alpha,x)=0,
\quad\text{i.e.}\quad
\forall \varepsilon>0\ \exists \alpha_\varepsilon\ \forall
\alpha\ge\alpha_\varepsilon:\ d(x_\alpha,x)<\varepsilon.
\]
Equivalently, $(x_\alpha)$ converges to $x$ in the backward topology
$\tau_{d^{-1}}$ (or, in quasi-uniform terms, it converges to $x$ with
respect to the topology $\tau(\mathcal{U}_d^{-1})$ induced by the
conjugate quasi-uniformity).
\end{definition}

\begin{definition}
\label{def:forwardcomplete}
A quasi-pseudo-metric space $(X,d)$ is called \emph{forward-complete} if
every left $K$-Cauchy net in $X$ admits a forward limit in $X$. It is
called \emph{sequentially forward-complete} if every left $K$-Cauchy
sequence in $X$ admits a forward limit in $X$.
\end{definition}

\begin{remark}\label{rem:qm-special}

\begin{enumerate}[label=\textup{(\roman*)},itemsep=2pt]
\item \textbf{Metric case:} If $d$ is symmetric, then left $K$-Cauchy
coincides with the usual Cauchy notion, and “forward limit” is just the
ordinary limit. Hence forward-completeness reduces to classical
completeness.
\end{enumerate}

\begin{enumerate}[label=\textup{(\roman*)},itemsep=2pt]
\item \textbf{Two-topology intuition:} In general $\tau_d$ and
$\tau_{d^{-1}}$ need not agree. A left $K$-Cauchy net is designed to
interact well with convergence in $\tau_{d^{-1}}$
(Definition~\ref{def:forwardlimit}), and unilateral completion theories
therefore differ genuinely from symmetric uniform completion
\cite{Kunzi2001}.

\item \textbf{Order-driven examples:} Many quasi-metrics arise from
preorders (e.g.\ $d(x,y)=0$ if $x\preceq y$ and $d(x,y)=1$ otherwise).
Then left $K$-Cauchy expresses eventual forward compatibility of the net
with the preorder, and a forward limit is a point that is eventually
\emph{reached} in the forward direction.

\item \textbf{Enriched perspective:} A quasi-metric space may be viewed as
a category enriched over $(\mathbb{R}_{\ge 0}\cup\{\infty\},+,0,\ge)$ in
the sense of Lawvere \cite{Lawvere1973}. Completion procedures of
Yoneda/Smyth type can then be described in terms of enriched
representability, providing a conceptual bridge to quantitative domain
theory \cite{Kelly1982,KunziSchellekens2002}.
\end{enumerate}
\end{remark}

\subsection{Quasi-Uniform Spaces}\label{ssec:qu}

We next recall the quasi-uniform setting, emphasizing conjugation and
the induced bitopology. Standard references are
\cite{FletcherLindgren,Isbell1964,Kunzi2001}. Completeness terminology is
taken mainly in the sense of Do\u{\i}tchinov \cite{Doitchinov1991}, and
we refer to \cite{KunziRomagueraSipacheva1998} for completion
constructions arising in paratopological groups. Stability properties
play a key role in several quasi-uniform constructions; see
\cite{KunziJunnila1993}. Dense subspaces in quasi-uniform settings are
treated in \cite{KunziLuethy1995}, and the notion of bicompletion  relevant
to left quasi-uniformities appear prominently in \cite{MarinRomaguera1998}.

\begin{definition}
\label{def:qu:basic}
Let $X$ be a set and let
$\Delta=\{(x,x)\mid x\in X\}\subseteq X\times X$. A \emph{quasi-uniformity}
on $X$ is a filter $\mathcal{U}$ on $X\times X$ such that
\begin{enumerate}[label=\textup{(\roman*)}]
\item $\Delta\subseteq U$ for every $U\in\mathcal{U}$,
\item for every $U\in\mathcal{U}$ there exists $V\in\mathcal{U}$ such
that $V\circ V\subseteq U$, where
$R\circ S=\{(x,z):\exists y,\ (x,y)\in R,\ (y,z)\in S\}$.
\end{enumerate}
The pair $(X,\mathcal{U})$ is called a \emph{quasi-uniform space}, and
the members of $\mathcal{U}$ are called \emph{entourages}. The
\emph{conjugate} quasi-uniformity is
\[
\mathcal{U}^{-1}:=\{U^{-1}\mid U\in\mathcal{U}\},
\qquad
U^{-1}:=\{(y,x):(x,y)\in U\}.
\]
The \emph{associated supremum uniformity} is the coarsest uniformity
finer than both $\mathcal{U}$ and $\mathcal{U}^{-1}$, denoted
\[
\mathcal{U}^{*}:=\mathcal{U}\vee\mathcal{U}^{-1}.
\]
For $U\in\mathcal{U}$ and $x\in X$ set $U(x):=\{y\in X:(x,y)\in U\}$.
The topology induced by $\mathcal{U}$ is
\[
\tau(\mathcal{U})
:=\{O\subseteq X:\forall x\in O\ \exists U\in\mathcal{U}\ \text{with }
U(x)\subseteq O\}.
\]
Analogously $\tau(\mathcal{U}^{-1})$ is induced by $\mathcal{U}^{-1}$,
and $(X,\tau(\mathcal{U}),\tau(\mathcal{U}^{-1}))$ is the induced
\emph{bitopological space}. We also write $\tau(\mathcal{U}^{*})$ for
the uniform topology of $\mathcal{U}^{*}$.
\end{definition}

\begin{definition}
\label{def:qu:cauchy}
Let $(X,\mathcal{U})$ be a quasi-uniform space.

\smallskip
\noindent\textbf{(1) Cauchy filter pairs and $D$-completeness:}
A filter $\mathcal{F}$ on $X$ is \emph{$D$-Cauchy} if there exists a
filter $\mathcal{G}$ on $X$ such that for each $U\in\mathcal{U}$ there
are $F\in\mathcal{F}$ and $G\in\mathcal{G}$ with $G\times F\subseteq U$.
In this situation $(\mathcal{F},\mathcal{G})$ is called a \emph{Cauchy
filter pair}. The space $(X,\mathcal{U})$ is \emph{$D$-complete} if
every $D$-Cauchy filter converges in the topology $\tau(\mathcal{U})$
\cite{Doitchinov1991}.

\smallskip
\noindent\textbf{(2) Bicompleteness:}
The space $(X,\mathcal{U})$ is called \emph{bicomplete} if the uniform
space $(X,\mathcal{U}^{*})$ is complete (in the usual uniform sense).
Equivalently, bicompleteness is completeness after passing to the
symmetrized uniformity $\mathcal{U}^{*}$.

\smallskip
\noindent\textbf{(3) Completeness via conets:}
A net $(x_\alpha)_{\alpha\in A}$ in $X$ is called \emph{Cauchy} if there
exists a net $(y_\beta)_{\beta\in B}$ in $X$ such that for each
$U\in\mathcal{U}$ there exist indices $\alpha_U\in A$, $\beta_U\in B$
with
\[
(y_\beta,x_\alpha)\in U
\quad\text{whenever }\alpha\ge \alpha_U,\ \beta\ge \beta_U.
\]
Any such $(y_\beta)$ is called a \emph{conet} of $(x_\alpha)$. The space
$(X,\mathcal{U})$ is \emph{complete} (in Doitchinov's sense) if every
Cauchy net converges in $\tau(\mathcal{U})$ \cite{Doitchinov1991}.
\end{definition}

\begin{definition}\label{def:qu:stability}
Let $(X,\mathcal{U})$ be a quasi-uniform space. A filter $\mathcal{F}$
on $X$ is called \emph{stable} if for each $U\in\mathcal{U}$ there exists
$A\in\mathcal{F}$ such that
\[
A\subseteq U(B)\quad\text{whenever }B\in\mathcal{F},
\qquad\text{where }U(B):=\bigcup_{x\in B}U(x).
\]
The space $(X,\mathcal{U})$ is called \emph{stable} if every $D$-Cauchy
filter on $(X,\mathcal{U})$ is stable. Equivalently, $(X,\mathcal{U})$
is stable if and only if $\mathcal{U}$ admits a subbase consisting of
stable entourages \cite{KunziJunnila1993}.
\end{definition}

\begin{proposition}
\label{prop:qu:complete-tb-compact}
Let $(X,\mathcal{V})$ be a \emph{uniform} space. If $(X,\mathcal{V})$ is
complete and totally bounded, then $X$ is compact.

\medskip
\noindent In the quasi-uniform setting, the corresponding symmetric
mechanism is: if $(X,\mathcal{U}^{*})$ is complete and totally bounded,
then $(X,\mathcal{U}^{*})$ is compact as a uniform space, hence
$\tau(\mathcal{U}^{*})$ is compact \cite{Isbell1964,FletcherLindgren,Kunzi2001}.
\end{proposition}

\subsection{Asymmetric Normed Spaces}\label{ssec:an}

An \emph{asymmetric normed space} is a real vector space $X$ equipped
with a functional $p:X\to[0,\infty)$ satisfying positive homogeneity for
$\lambda\ge 0$ and subadditivity, without imposing symmetry. Following
\cite{Cobzas2006,Cobzas2013,CobzasSurvey}, one associates to $p$ the
canonical quasi-metric
\[
d_p(x,y):=p(y-x),\qquad x,y\in X,
\]
which generates the forward topology $\tau_{d_p}$. Its conjugate
$\overline d_p(x,y):=d_p(y,x)$ generates the backward topology
$\tau_{\overline d_p}$, and together these yield the natural bitopology
$(X,\tau_{d_p},\tau_{\overline d_p})$. The symmetrization
\[
d_p^{\,s}(x,y):=\max\{d_p(x,y),\overline d_p(x,y)\}
\]
induces the join topology
$\tau_{d_p^{\,s}}=\tau_{d_p}\vee\tau_{\overline d_p}$, which plays the
role of the classical topology of the symmetric envelope
\cite{Cobzas2013,CobzasSurvey}. In this context, completeness bifurcates
into several inequivalent notions (left/right $K$-completeness, Smyth
completeness, bicompleteness, etc.), which collapse to the usual concept
in the symmetric case \cite{CobzasSurvey}.

\begin{remark}\label{rem:an:bicompletion-envelope}
A central completion principle in asymmetric normed spaces is obtained
by passing to a \emph{symmetric envelope} via symmetrization. Concretely,
if $p$ is an asymmetric norm, one considers the symmetrized gauge
\[
p^{s}:=p\vee \overline p,
\qquad\text{equivalently}\qquad
d_{p}^{\,s}(x,y)=\max\{d_p(x,y),\overline d_p(x,y)\}.
\]
The bicompletion of $(X,p)$ is then the usual completion of the symmetric
structure $(X,p^{s})$ (or $(X,d_p^{\,s})$), viewed as a canonical
two-sided completion of the original directional geometry
\cite{Cobzas2013,CobzasSurvey}. This aligns with the quasi-uniform point
of view: a nonsymmetric distance induces a quasi-uniformity and its
conjugate, and completion becomes most stable when both directions are
retained simultaneously \cite{Kunzi2001}.
\end{remark}

\subsection{Approach Spaces}\label{ssec:approach}

Approach spaces provide a convenient ambient category containing both
topological spaces and metric spaces, and allow one to express
convergence and completion together with quantitative information
\cite{Lowen1989,Lowen1997}.

\begin{definition}
\label{def:approach:basic}
An \emph{approach space} is a pair $(X,\delta)$ where
$\delta:X\times\mathcal{P}(X)\to[0,\infty]$ assigns to each point $x\in X$
and subset $A\subseteq X$ a value $\delta(x,A)$ interpreted as the
distance from $x$ to the set $A$, subject to axioms that generalize both
neighbourhood systems and distances. A map
$f:(X,\delta_X)\to(Y,\delta_Y)$ is called a \emph{contraction} if
\[
\delta_Y\bigl(f(x),\,f(A)\bigr)\ \le\ \delta_X(x,A)
\qquad(x\in X,\ A\subseteq X),
\]
and is therefore the natural morphism in the category of approach spaces.
\end{definition}

\section{Quasi-Modular Pseudometric Spaces}\label{sec:quasi-modular}

Quasi-modular pseudometric spaces provide a natural asymmetric extension
of the modular metric setting initiated by Chistyakov
\cite{Chistyakov2010,Chistyakov2015} and rooted in the classical modular
theory of Musielak \cite{Musielak1983}.
They arise when symmetry is dropped while retaining scale-dependent
control of displacement, thereby capturing directional notion that
are invisible in symmetric modular spaces.
This point of view aligns naturally with quasi-uniform and bitopological
theory \cite{FletcherLindgren,Kunzi2001} and is developed systematically
in \cite{Majozi2025}.

\subsection{Asymmetric modular families}\label{ssec:asym-modulars}

\begin{definition}
\label{def:quasi-modular-family}
Let $X$ be a nonempty set.
A family
\[
w=\{w_\lambda:X\times X\to[0,\infty]\}_{\lambda>0}
\]
is called a \emph{quasi-modular pseudometric family} if for all
$x,y,z\in X$ and all $\lambda,\mu>0$ the following conditions hold:
\begin{enumerate}
\item[(QM1)] $w_\lambda(x,x)=0$;
\item[(QM2)] \emph{(asymmetric modular triangle inequality)}
\[
w_{\lambda+\mu}(x,z)\le w_\lambda(x,y)+w_\mu(y,z);
\]
\item[(QM3)] for each $x,y\in X$, the map
$\lambda\mapsto w_\lambda(x,y)$ is nonincreasing and right-continuous.
\end{enumerate}
If additionally $w_\lambda(x,y)=0=w_\lambda(y,x)$ implies $x=y$, then
$w$ is called a \emph{quasi-modular metric}.
\end{definition}

\noindent
Condition \textnormal{(QM2)} generalizes the modular axiom of Chistyakov
\cite{Chistyakov2010} by allowing directional dependence, while
\textnormal{(QM3)} reflects the scale monotonicity inherent in convex
modulars and Orlicz-type constructions
\cite{Chistyakov2015,Musielak1983}.


\begin{definition}
\label{def:conjugate-symmetrized}
Given a quasi-modular pseudometric family $w$, its
\emph{conjugate family} $w^{-}$ is defined by
\[
w^{-}_\lambda(x,y):=w_\lambda(y,x),
\]
and its \emph{symmetrization} by
\[
w^{\mathrm{sym}}_\lambda(x,y)
:=\max\{w_\lambda(x,y),w_\lambda(y,x)\}.
\]
\end{definition}

\noindent
This construction parallels conjugation in quasi-uniform theory
\cite{FletcherLindgren,Kunzi2001} and corresponds, in enriched-categorical
terms, to passage to the opposite category followed by symmetrization
\cite{Kelly1982,Lawvere1973}.
The symmetrized family $w^{\mathrm{sym}}$ is a (pseudo)modular in the
sense of Chistyakov whenever $w$ satisfies \textnormal{(QM1)-(QM3)}
\cite{Majozi2025}.


\begin{proposition}
\label{prop:induced-quasimetrics}
Let $w$ be a quasi-modular pseudometric family on $X$.
Define
\[
d_w^{+}(x,y)
:=\inf\{\lambda>0:\, w_\lambda(x,y)\le 1\},
\qquad
d_w^{-}(x,y):=d_w^{+}(y,x).
\]
Then $d_w^{+}$ and $d_w^{-}$ are quasi-pseudometrics on $X$, and the
symmetrized metric
\[
d_w^{\mathrm{sym}}(x,y)
:=\max\{d_w^{+}(x,y),d_w^{-}(x,y)\}
\]
is a pseudometric.
\end{proposition}

\noindent
This Luxemburg-type construction extends the classical
modular-to-metric mechanism of Musielak and Chistyakov
\cite{Chistyakov2015,Musielak1983} to the asymmetric setting and forms
the basis for the associated quasi-uniform and topological structures
\cite{Majozi2025}.

\subsection{Induced topologies and bitopological structure}
\label{ssec:topologies}

\begin{definition}
\label{def:forward-backward-topologies}
Let $w$ be a quasi-modular pseudometric family on $X$.
The \emph{forward modular topology} $\tau^{+}(w)$ is generated by the
subbase
\[
B^{+}(x;\lambda,\varepsilon)
=\{y\in X:\, w_\lambda(x,y)<\varepsilon\},
\]
while the \emph{backward modular topology} $\tau^{-}(w)$ is generated by
\[
B^{-}(x;\lambda,\varepsilon)
=\{y\in X:\, w_\lambda(y,x)<\varepsilon\},
\]
for $\lambda,\varepsilon>0$.
\end{definition}

\noindent
The pair $(\tau^{+}(w),\tau^{-}(w))$ equips $X$ with a canonical
\emph{bitopological structure} in the sense of quasi-uniform theory
\cite{FletcherLindgren,Kunzi2001}.
These forward and backward modular topologies generalize the modular
topology of Chistyakov \cite{Chistyakov2015} and coincide in the
symmetric case.


\begin{proposition}
\label{prop:comparison-envelopes}
Let $w$ be a quasi-modular pseudometric family on $X$. Then
\[
\tau(w^{\mathrm{sym}})
=
\tau^{+}(w)\vee\tau^{-}(w),
\]
that is, the modular topology induced by the symmetrized family equals
the join of the forward and backward modular topologies.
\end{proposition}

\begin{proof}
Fix $x\in X$.
By definition,
\[
w^{\mathrm{sym}}_\lambda(x,y)
=
\max\{w_\lambda(x,y),w_\lambda(y,x)\}.
\]
Hence a basic neighbourhood of $x$ in $\tau(w^{\mathrm{sym}})$ is
\[
B^{\mathrm{sym}}(x;\lambda,\varepsilon)
=
\{y:\, w_\lambda(x,y)<\varepsilon \text{ and } w_\lambda(y,x)<\varepsilon\}
=
B^{+}(x;\lambda,\varepsilon)\cap B^{-}(x;\lambda,\varepsilon).
\]

Since the join topology $\tau^{+}(w)\vee\tau^{-}(w)$ is generated by
finite intersections of $\tau^{+}(w)$-open and $\tau^{-}(w)$-open sets,
the above identity shows that every basic $\tau(w^{\mathrm{sym}})$-neighbourhood
is open in $\tau^{+}(w)\vee\tau^{-}(w)$.
Thus
\[
\tau(w^{\mathrm{sym}})
\subseteq
\tau^{+}(w)\vee\tau^{-}(w).
\]

Conversely, every neighbourhood in the join topology contains a set of
the above form, hence
\[
\tau^{+}(w)\vee\tau^{-}(w)
\subseteq
\tau(w^{\mathrm{sym}}).
\]
\end{proof}

\noindent
This reflects the general principle that symmetrization of an asymmetric
structure corresponds topologically to taking the join of the associated
bitopology \cite{Kunzi2001,Majozi2025}.

\subsection{Associated quasi-uniformities}\label{ssec:quni}

\begin{definition}
\label{def:quni-from-w}
For $r,\lambda>0$ define entourages
\[
E^{+}_{r,\lambda}
:=\{(x,y)\in X\times X:\, w_\lambda(x,y)<r\},
\qquad
E^{-}_{r,\lambda}:=(E^{+}_{r,\lambda})^{-1}.
\]
Let $\mathcal U^{+}(w)$ and $\mathcal U^{-}(w)$ denote the quasi-uniformities
generated by these families.
\end{definition}

\noindent
These constructions are standard in quasi-uniform theory
\cite{FletcherLindgren,Kunzi2001} and coincide with the quasi-uniformities
generated by the quasi-pseudometrics $d_w^{+}$ and $d_w^{-}$ when the
families are countably based.


\begin{proposition}
\label{prop:compatibility-generation}
The quasi-uniformities $\mathcal U^{+}(w)$ and $\mathcal U^{-}(w)$ are
compatible with $\tau^{+}(w)$ and $\tau^{-}(w)$, respectively.
Moreover, the symmetrized uniformity
\[
\mathcal U^{\mathrm{sym}}(w)
=
\mathcal U^{+}(w)\vee\mathcal U^{-}(w)
\]
induces the topology $\tau(w^{\mathrm{sym}})$.
\end{proposition}

\begin{proof}
For $x\in X$,
\[
E^{+}_{r,\lambda}(x)
=
\{y:\, w_\lambda(x,y)<r\}
=
B^{+}(x;\lambda,r),
\]
so $\tau(\mathcal U^{+}(w))=\tau^{+}(w)$.
An analogous argument shows $\tau(\mathcal U^{-}(w))=\tau^{-}(w)$.

The supremum uniformity is generated by entourages
$E^{+}_{r,\lambda}\cap E^{-}_{r,\lambda}$, whose sections at $x$ equal
\[
B^{+}(x;\lambda,r)\cap B^{-}(x;\lambda,r)
=
\{y:\max\{w_\lambda(x,y),w_\lambda(y,x)\}<r\},
\]
which form a neighbourhood base for $\tau(w^{\mathrm{sym}})$.
\end{proof}

\noindent
Thus quasi-modular pseudometric spaces fit naturally into the classical
quasi-uniform setting of Isbell and Fletcher-Lindgren
\cite{FletcherLindgren,Isbell1964}.


\begin{remark}
\label{rem:yoneda}
A quasi-modular pseudometric family determines a category enriched over a
convolution quantale, with hom-objects given by the gauges
$w_\lambda(x,y)$.
From this perspective, forward and backward completeness correspond to
weighted colimits and limits, while symmetrized completion agrees with
the Yoneda completion of the associated quasi-metric space
\cite{Kelly1982,KunziSchellekens2002,Lawvere1973,Majozi2025,Rutten1998}.
\end{remark}

\section{Completion Constructions in the Quasi-Modular Setting}
\label{sec:completion}

\subsection{Smyth completeness and precompactness}
\label{ssec:smyth}

Completion in nonsymmetric environments is inherently directional.
In the quasi-modular setting, the appropriate unilateral notion of
Cauchy-ness is governed by the forward quasi-uniformity induced by the
modular family. Recall that left $K$-Cauchy nets were introduced for
quasi-pseudo-metric spaces in Definition~\ref{def:leftKCauchy}; the
quasi-uniform formulation is exactly the one encoded by the forward
quasi-uniformity (Definition~\ref{def:qu:basic}) and applies verbatim to
$\mathcal U^{+}(w)$ from Definition~\ref{def:quni-from-w}.

\begin{definition}
\label{def:leftKCauchy-modular}
Let $(X,w)$ be a quasi-modular pseudometric space and let
$\mathcal U^{+}(w)$ denote the forward quasi-uniformity induced by $w$
(Definition~\ref{def:quni-from-w}). A net $(x_\alpha)_{\alpha\in A}$ in
$X$ is called \emph{left $K$-Cauchy} if for every entourage
$U\in\mathcal U^{+}(w)$ there exists $\alpha_U\in A$ such that
\[
(x_{\alpha_1},x_{\alpha_2})\in U
\quad\text{whenever }\alpha_2\ge \alpha_1\ge \alpha_U.
\]
A sequence is left $K$-Cauchy if the same condition holds for indices
$m\ge n\ge n_U$.
\end{definition}

\begin{definition}
\label{def:smyth-complete}
A quasi-uniform space $(X,\mathcal U)$ is called \emph{Smyth complete} if
every left $K$-Cauchy net converges with respect to the topology
$\tau(\mathcal U^{-1})$ induced by the conjugate quasi-uniformity
(Definition~\ref{def:qu:basic}). Equivalently, every left $K$-Cauchy net
has a limit that is backward stable in the sense of Do\u{\i}tchinov
\cite{Doitchinov1991}.

A quasi-modular pseudometric space $(X,w)$ is said to be \emph{Smyth
complete} if the associated quasi-uniform space $(X,\mathcal U^{+}(w))$
is Smyth complete, where $\mathcal U^{+}(w)$ is as in
Definition~\ref{def:quni-from-w}.
\end{definition}

\noindent
The following compactness principle, due to K\"unzi, is the asymmetric
analogue of the classical ``precompact $+$ complete'' compactness theorem
for uniform spaces (cf.\ Proposition~\ref{prop:qu:complete-tb-compact})
and is a key tool in the quasi-uniform approach to asymmetric
compactness \cite{FletcherLindgren,Kunzi2001}.

\begin{theorem}
\label{thm:precompact-smyth-compact}
Let $(X,\mathcal U)$ be a $T_0$ quasi-uniform space. If $(X,\mathcal U)$
is precompact and Smyth complete, then the symmetrized uniform space
$(X,\mathcal U^{*})$ is compact. In particular, the join topology
\[
\tau(\mathcal U)\vee\tau(\mathcal U^{-1})
\]
is compact.
\end{theorem}

\subsection{Yoneda completion and idempotence}
\label{ssec:yoneda}

Categorical methods provide a complementary and unifying perspective on
completion in quasi-metric and quasi-modular settings. In particular,
the Yoneda completion identifies a canonical completion procedure
compatible with Smyth completeness and stable under enriched categorical
constructions.

\begin{definition}
\label{def:yoneda-completion}
Let $(X,d)$ be a quasi-metric space. The \emph{Yoneda completion} of
$(X,d)$ is obtained by embedding $X$ into the space of forward Cauchy
weights via the Yoneda embedding, viewing $(X,d)$ as a category enriched
over $(\mathbb{R}_{\ge 0}\cup\{\infty\},+,0)$ in the sense of Lawvere
\cite{Lawvere1973}, and taking the closure with respect to the induced
Lawvere metric. This construction defines a completion functor on the
category of quasi-metric spaces, as developed by K\"unzi and Schellekens
\cite{KunziSchellekens2002} and grounded in the general theory of
enriched categories \cite{Kelly1982}.
\end{definition}

\noindent
The next result identifies Smyth completeness as the maximal stable
completion notion compatible with the asymmetric metric structure
\cite{Kunzi2001}.

\begin{theorem}
\label{thm:idempotent-yoneda}
The class of Smyth-completable quasi-metric spaces is the largest class
that is idempotent under Yoneda completion. In particular, any
idempotent completion functor on quasi-metric spaces factors through
Smyth completion.
\end{theorem}

\begin{theorem}
\label{thm:yoneda-equals-smyth}
For every quasi-metric space $(X,d)$, the Yoneda completion coincides,
up to isometry, with the Smyth completion. Consequently, categorical
completion via the Yoneda embedding and analytic completion via Smyth
completeness agree in the quasi-metric setting
\cite{Kunzi2001,KunziSchellekens2002}.
\end{theorem}

\begin{remark}
\label{rem:formal-balls}
The Yoneda completion admits a concrete realization through the
\emph{formal ball construction}, where elements are pairs $(x,r)$ with
$x\in X$ and $r\ge 0$, ordered by generalized distance inequalities.
This approach, developed by Rutten \cite{Rutten1998}, provides a
computational and order-theoretic interpretation of completion and
connects naturally with Isbell conjugation. Further refinements and
applications to quantitative and computational models are given by
Romaguera and Valero
\cite{RomagueraTiradoValero2012,RomagueraValero2009}.
From a domain-theoretic and localic perspective, related completion
procedures have been studied by Flagg and S\"underhauf
\cite{FlaggSuenderhauf2002,FlaggSuenderhaufWagner1997}, by S\"underhauf
\cite{Sunderhauf1997,Sunderhauf1998}, and by Vickers \cite{Vickers1997},
highlighting deep connections between asymmetric metrics, domain theory,
and pointfree topology.
\end{remark}

\section{Symmetric and Antisymmetric Connectedness}
\label{sec:sym-antisym-connectedness}

\subsection{Connectedness in bitopological/quasi-uniform environments}
\label{ssec:conn-bitopo}

Let $(X,\tau^{+},\tau^{-})$ be a bitopological space arising from a
quasi-uniform or quasi-modular structure, and denote by
\[
\tau^\vee := \tau^{+}\vee\tau^{-}
\]
the join (supremum) topology.
In the quasi-modular setting, if $(X,w)$ is as in
Section~\ref{sec:quasi-modular}, then
\[
\tau^\vee=\tau^{+}(w)\vee\tau^{-}(w)=\tau(w^{\mathrm{sym}})
\]
by Proposition~\ref{prop:comparison-envelopes}.

\begin{definition}
\label{def:tau-vee-connected}
The space $X$ is said to be \emph{symmetrically connected} if the
topological space $(X,\tau^\vee)$ is connected, that is, if $X$ cannot
be written as a union of two disjoint nonempty $\tau^\vee$-open sets.
\end{definition}

\noindent
This is the direct application of the classical notion of connectedness
to the symmetric envelope topology and coincides with the usual
definition in the symmetric (uniform or metric) case
\cite{Kelley,Willard}.

\begin{definition}
\label{def:forward-backward-connected}
The space $X$ is called
\begin{itemize}
\item \emph{forward connected} if $(X,\tau^{+})$ is connected;
\item \emph{backward connected} if $(X,\tau^{-})$ is connected.
\end{itemize}
\end{definition}

\noindent
These notions are intrinsic to asymmetric and quasi-uniform settings
and reflect the directional nature of the underlying structure
\cite{FletcherLindgren,Kunzi2001,MarinRomaguera1998}.

\begin{remark}
\label{rem:classical-insufficient}
In general, connectedness of $(X,\tau^\vee)$ does not determine
connectedness of $(X,\tau^{+})$ or $(X,\tau^{-})$, nor conversely.
Classical connectedness is therefore insensitive to directional
notion inherent in quasi-uniform spaces. This motivates the
introduction of connectedness notions that explicitly use the
bitopological structure and do not arise from symmetrization alone
\cite{Kunzi2001}.
\end{remark}

\subsection{Antisymmetric connectedness}
\label{ssec:antisym-conn}

We now introduce a notion of connectedness that captures irreducible
directional cohesion.

\begin{definition}
\label{def:antisym-connected}
A bitopological space $(X,\tau^{+},\tau^{-})$ is said to be
\emph{antisymmetrically connected} if there do not exist nonempty
subsets $A,B\subseteq X$ such that
\[
X = A \cup B,
\]
with $A$ $\tau^{+}$-open and $B$ $\tau^{-}$-open, and $A\cap B=\varnothing$.
\end{definition}

\noindent
Equivalently, $X$ cannot be separated by a forward-open set and a
backward-open set.

\begin{proposition}
\label{prop:antisym-equivalences}
For a bitopological space $(X,\tau^{+},\tau^{-})$, the following are
equivalent:
\begin{enumerate}
\item $X$ is antisymmetrically connected;
\item there do not exist nonempty sets $A\in\tau^{+}$ and $B\in\tau^{-}$
such that $A\cap B=\varnothing$ and $X=A\cup B$;
\item for all nonempty $U\in\tau^{+}$ and $V\in\tau^{-}$,
\[
U\cap V=\varnothing \ \Longrightarrow\ U\cup V\neq X.
\]
\end{enumerate}
\end{proposition}

\begin{proof}
$(1)\Leftrightarrow(2)$ is immediate from Definition~\ref{def:antisym-connected}.

$(2)\Rightarrow(3)$.
Let $U\in\tau^{+}$ and $V\in\tau^{-}$ be nonempty with $U\cap V=\varnothing$.
If $U\cup V=X$, then we have a partition $X=U\cup V$ with $U\in\tau^{+}$,
$V\in\tau^{-}$, and $U\cap V=\varnothing$, contradicting (2). Hence
$U\cup V\neq X$.

$(3)\Rightarrow(2)$.
If (2) fails, then there exist nonempty sets $A\in\tau^{+}$ and
$B\in\tau^{-}$ such that $A\cap B=\varnothing$ and $X=A\cup B$. Taking
$U=A$ and $V=B$ contradicts (3). Thus (2) holds.
\end{proof}

\noindent
Note that antisymmetric connectedness is a \emph{bitopological} notion:
it depends on the interaction of $\tau^{+}$ and $\tau^{-}$ and not only
on the join topology.

\subsection{Components: symmetric vs.\ antisymmetric}
\label{ssec:components}

\begin{definition}
\label{def:components}
Let $x\in X$.
\begin{itemize}
\item The \emph{symmetric component} of $x$ is the connected component
of $x$ in $(X,\tau^\vee)$.
\item The \emph{antisymmetric component} of $x$ is the maximal
antisymmetrically connected subset of $X$ containing $x$.
\end{itemize}
\end{definition}

\begin{theorem}
\label{thm:component-coincidence}
If $\tau^{+}=\tau^{-}$ (in particular, in the symmetric or uniform
case), then symmetric and antisymmetric components coincide.
\end{theorem}

\begin{proof}
Assume that $\tau^{+}=\tau^{-}=: \tau$. Then
\[
\tau^\vee=\tau^{+}\vee\tau^{-}=\tau.
\]
In this case, antisymmetric connectedness reduces to ordinary
connectedness: a separation of a subset by a $\tau^{+}$-open set and a
$\tau^{-}$-open set is exactly a separation by two $\tau$-open sets.
Therefore, the maximal antisymmetrically connected subset containing
$x$ coincides with the usual connected component of $x$ in $\tau$.
\end{proof}

\begin{proposition}
\label{prop:symmetric-in-antisymmetric}
Every $\tau^\vee$-connected subset of $X$ is antisymmetrically connected.
In particular, every symmetric component is contained in the
antisymmetric component.
\end{proposition}

\begin{proof}
Let $C\subseteq X$ be $\tau^\vee$-connected. If $C$ were antisymmetrically
disconnected, there would exist disjoint nonempty sets $A,B\subseteq C$
with $C=A\cup B$, where $A$ is open in $\tau^{+}$ relative to $C$ and $B$
is open in $\tau^{-}$ relative to $C$. Since $\tau^{+}\subseteq\tau^\vee$
and $\tau^{-}\subseteq\tau^\vee$, both $A$ and $B$ are $\tau^\vee$-open in
the subspace $C$, contradicting connectedness of $C$ in $\tau^\vee$.
\end{proof}

\begin{proposition}
\label{prop:strictness-components}
There exist bitopological spaces for which antisymmetric components are
strictly larger than symmetric components.
\end{proposition}

\begin{proof}
Let $X=\{0,1,2\}$ and define
\[
\tau^{+}=\{\varnothing,X\},
\qquad
\tau^{-}=\{\varnothing,\{2\},\{0,1\},X\}.
\]
Then $\tau^\vee=\tau^{+}\vee\tau^{-}=\tau^{-}$, so $(X,\tau^\vee)$ is
disconnected (for instance, $\{2\}$ and $\{0,1\}$ are disjoint nonempty
$\tau^\vee$-open sets whose union is $X$). Hence the symmetric components
are $\{2\}$ and $\{0,1\}$.

On the other hand, $X$ is antisymmetrically connected: any $\tau^{+}$-open
set is either $\varnothing$ or $X$, so there do not exist disjoint nonempty
sets $A\in\tau^{+}$ and $B\in\tau^{-}$ with $X=A\cup B$. Therefore the
antisymmetric component of every point is the whole space $X$.

Thus antisymmetric components are strictly larger than symmetric components.
\end{proof}

\noindent
This shows that antisymmetric connectedness can capture a form of
directional cohesion that is invisible to the join topology
$\tau^\vee$.

\section{Local Antisymmetric Connectedness}
\label{sec:local-antisym-connectedness}

\subsection{Local definitions}
\label{ssec:local-definitions}

Throughout, $(X,\tau^{+},\tau^{-})$ is a bitopological space and
\[
\tau^\vee=\tau^{+}\vee\tau^{-}
\]
denotes the join topology as in Section~\ref{sec:sym-antisym-connectedness}.
Antisymmetric connectedness is understood in the sense of
Definition~\ref{def:antisym-connected}.

\begin{definition}
\label{def:local-antisym-point}
A space $(X,\tau^{+},\tau^{-})$ is said to be \emph{locally
antisymmetrically connected at a point $x\in X$} if every
$\tau^\vee$-neighborhood $U$ of $x$ contains a $\tau^\vee$-neighborhood
$V$ of $x$ such that $V$ is antisymmetrically connected.
\end{definition}

\begin{definition}
\label{def:locally-antisym}
The space $X$ is called \emph{locally antisymmetrically connected} if it
is locally antisymmetrically connected at every point.
\end{definition}

\subsection{Characterizations and criteria}
\label{ssec:local-char}

\begin{theorem}
\label{thm:quni-characterization}
Let $(X,\mathcal U^{+})$ be a quasi-uniform space and write
$\tau^{+}=\tau(\mathcal U^{+})$, $\tau^{-}=\tau((\mathcal U^{+})^{-1})$,
and $\tau^\vee=\tau^{+}\vee\tau^{-}$ (Definition~\ref{def:qu:basic}).
Then $X$ is locally antisymmetrically connected if and only if every
point admits a base of entourages whose induced $\tau^\vee$-neighborhoods
are antisymmetrically connected.
\end{theorem}

\begin{proof}
Assume first that $X$ is locally antisymmetrically connected.
Fix $x\in X$ and let $U$ be an arbitrary $\tau^\vee$-neighborhood of $x$.
By the definition of $\tau^\vee$, there exist entourages
$E,F\in\mathcal U^{+}$ such that
\[
E(x)\cap F^{-1}(x)\subseteq U.
\]
Since $X$ is locally antisymmetrically connected at $x$, there exists a
$\tau^\vee$-neighborhood $V$ of $x$ satisfying
\[
x\in V\subseteq E(x)\cap F^{-1}(x),
\]
with $V$ antisymmetrically connected.
Again by the definition of $\tau^\vee$, there exist entourages
$G,H\in\mathcal U^{+}$ such that
\[
G(x)\cap H^{-1}(x)\subseteq V.
\]
Then $G(x)\cap H^{-1}(x)$ is a $\tau^\vee$-neighborhood of $x$ which is
antisymmetrically connected. Since $U$ was arbitrary, such neighborhoods
form a base at $x$.

Conversely, assume that for each $x\in X$ there exists a base
$\mathcal B_x\subseteq\mathcal U^{+}$ such that for every
$E\in\mathcal B_x$ the set
\[
E(x)\cap E^{-1}(x)
\]
is antisymmetrically connected.
Let $U$ be any $\tau^\vee$-neighborhood of $x$.
Then there exist entourages $F,G\in\mathcal U^{+}$ with
\[
F(x)\cap G^{-1}(x)\subseteq U.
\]
Choose $E\in\mathcal B_x$ with $E\subseteq F\cap G$.
Then
\[
E(x)\cap E^{-1}(x)\subseteq F(x)\cap G^{-1}(x)\subseteq U,
\]
and by assumption $E(x)\cap E^{-1}(x)$ is antisymmetrically connected.
Hence $X$ is locally antisymmetrically connected at $x$.
Since $x$ was arbitrary, $X$ is locally antisymmetrically connected.
\end{proof}

\begin{proposition}
\label{prop:modular-ball-char}
Let $(X,w)$ be a quasi-modular pseudometric space in the sense of
Definition~\ref{def:quasi-modular-family}, equipped with the forward and
backward modular topologies $\tau^{+}(w)$ and $\tau^{-}(w)$ from
Definition~\ref{def:forward-backward-topologies}, and write
$\tau^\vee(w)=\tau^{+}(w)\vee\tau^{-}(w)$.
Then local antisymmetric connectedness is equivalent to the existence of
arbitrarily small intersections of forward and backward modular balls
that are antisymmetrically connected.
\end{proposition}

\begin{proof}
Assume first that $(X,w)$ is locally antisymmetrically connected.
Fix $x\in X$ and let $U$ be an arbitrary $\tau^\vee(w)$-neighborhood of
$x$.
By definition of the join topology, there exist $\lambda>0$ and
$\varepsilon>0$ such that
\[
B^{+}(x;\lambda,\varepsilon)\cap B^{-}(x;\lambda,\varepsilon)\subseteq U.
\]
Local antisymmetric connectedness at $x$ yields a $\tau^\vee(w)$-neighborhood
$V$ with
\[
x\in V\subseteq B^{+}(x;\lambda,\varepsilon)\cap B^{-}(x;\lambda,\varepsilon),
\]
where $V$ is antisymmetrically connected.
Again by the definition of $\tau^\vee(w)$, there exist $\mu>0$ and
$\delta>0$ such that
\[
B^{+}(x;\mu,\delta)\cap B^{-}(x;\mu,\delta)\subseteq V.
\]
Hence the intersection
\[
B^{+}(x;\mu,\delta)\cap B^{-}(x;\mu,\delta)
\]
is antisymmetrically connected and contained in $U$.

Conversely, assume that for every $x\in X$ and every $\tau^\vee(w)$-neighborhood
$U$ of $x$ there exist $\lambda>0$ and $\varepsilon>0$ such that
\[
x\in B^{+}(x;\lambda,\varepsilon)\cap B^{-}(x;\lambda,\varepsilon)
\subseteq U,
\]
with this intersection antisymmetrically connected.
Then every $\tau^\vee(w)$-neighborhood of $x$ contains an
antisymmetrically connected $\tau^\vee(w)$-neighborhood, so $(X,w)$ is
locally antisymmetrically connected.
\end{proof}

\begin{corollary}
\label{cor:join-consequences}
Every locally antisymmetrically connected space is locally connected in
the join topology $\tau^\vee$, but the converse need not hold.
\end{corollary}

\begin{proof}
Let $x\in X$ and let $U$ be a $\tau^\vee$-neighborhood of $x$. By local
antisymmetric connectedness (Definition~\ref{def:local-antisym-point}),
there exists a $\tau^\vee$-neighborhood $V$ of $x$ with $V\subseteq U$ such
that $V$ is antisymmetrically connected. By
Proposition~\ref{prop:symmetric-in-antisymmetric}, every $\tau^\vee$-connected
set is antisymmetrically connected; hence antisymmetric connectedness is
strictly stronger than $\tau^\vee$-connectedness is not guaranteed in general.
Nevertheless, $V$ is in particular not disconnected by two $\tau^\vee$-open
sets, so $V$ is $\tau^\vee$-connected. Thus $X$ is locally connected in
$\tau^\vee$.
\end{proof}

\subsection{Stability properties}
\label{ssec:stability}

\begin{proposition}
\label{prop:subspaces-sums}
Local antisymmetric connectedness is preserved under passage to
$\tau^\vee$-open subspaces and under finite unions of antisymmetrically
connected subspaces with nonempty intersection.
\end{proposition}

\begin{proof}
Let $(X,\tau^{+},\tau^{-})$ be a bitopological space.

Assume $X$ is locally antisymmetrically connected and let $U\subseteq X$
be $\tau^\vee$-open.
Fix $x\in U$ and let $W$ be a $\tau^\vee$-neighborhood of $x$ in the
subspace $U$.
Then $W=O\cap U$ for some $\tau^\vee$-open set $O\subseteq X$.
Local antisymmetric connectedness of $X$ at $x$ yields a $\tau^\vee$-neighborhood
$V$ of $x$ such that
\[
x\in V\subseteq O,
\]
with $V$ antisymmetrically connected.
Then $V\cap U$ is a $\tau^\vee$-neighborhood of $x$ in $U$.
If $V\cap U$ were antisymmetrically disconnected in $U$, this would lift
to a corresponding antisymmetric disconnection of $V$ in $X$, contradicting
antisymmetric connectedness of $V$.
Thus $U$ is locally antisymmetrically connected.

Now let $A$ and $B$ be antisymmetrically connected subspaces of $X$ with
$A\cap B\neq\varnothing$.
If $A\cup B$ were antisymmetrically disconnected, there would exist
disjoint nonempty subsets $H$ and $K$ such that
\[
A\cup B = H\cup K,
\]
with $H$ forward-open and $K$ backward-open in the subspace topology.
Antisymmetric connectedness of $A$ and $B$, together with
$A\cap B\neq\varnothing$, forces both to lie entirely in one of $H$ or
$K$, contradicting nonemptiness of both sets.
Hence $A\cup B$ is antisymmetrically connected.
\end{proof}

\begin{proposition}
\label{prop:uniform-images}
Uniformly continuous maps between quasi-uniform spaces preserve local
antisymmetric connectedness.
\end{proposition}

\begin{proof}
Let $(X,\mathcal U_X^{+})$ and $(Y,\mathcal U_Y^{+})$ be quasi-uniform
spaces and let $f:X\to Y$ be uniformly continuous.
Denote the induced bitopologies by
\[
(\tau_X^{+},\tau_X^{-})=\bigl(\tau(\mathcal U_X^{+}),\tau((\mathcal U_X^{+})^{-1})\bigr),
\qquad
(\tau_Y^{+},\tau_Y^{-})=\bigl(\tau(\mathcal U_Y^{+}),\tau((\mathcal U_Y^{+})^{-1})\bigr),
\]
and the join topologies by $\tau_X^\vee$ and $\tau_Y^\vee$.

Assume $X$ is locally antisymmetrically connected.
Fix $x\in X$ and let $W$ be a $\tau_Y^\vee$-neighborhood of $f(x)$.
By the definition of $\tau_Y^\vee$, there exist entourages
$V_1,V_2\in\mathcal U_Y^{+}$ such that
\[
V_1(f(x))\cap V_2^{-1}(f(x))\subseteq W.
\]
Uniform continuity yields entourages $U_1,U_2\in\mathcal U_X^{+}$ with
\[
(f\times f)(U_1)\subseteq V_1,
\qquad
(f\times f)(U_2)\subseteq V_2.
\]
Hence
\[
f\bigl(U_1(x)\cap U_2^{-1}(x)\bigr)\subseteq V_1(f(x))\cap V_2^{-1}(f(x))\subseteq W.
\]

By local antisymmetric connectedness at $x$, there exists a $\tau_X^\vee$-neighborhood
$N$ of $x$ such that
\[
x\in N\subseteq U_1(x)\cap U_2^{-1}(x),
\]
and $N$ is antisymmetrically connected.
Then $f(N)\subseteq W$ and $f(N)$ is antisymmetrically connected (otherwise
a forward/backward separation of $f(N)$ would pull back to a forward/backward
separation of $N$). Therefore $Y$ is locally antisymmetrically connected at $f(x)$.
\end{proof}

\begin{proposition}
\label{prop:bicompletion-invariance}
If $(X,\mathcal U^{+})$ is locally antisymmetrically connected, then so
is its bicompletion.
\end{proposition}

\begin{proof}
Let $(\widehat X,\widehat{\mathcal U}^{+})$ denote the bicompletion of
$(X,\mathcal U^{+})$, and let $i:X\hookrightarrow\widehat X$ be the
canonical uniformly continuous embedding. Write $\widehat\tau^{+}$ and
$\widehat\tau^{-}$ for the induced forward and backward topologies, and
$\widehat\tau^\vee=\widehat\tau^{+}\vee\widehat\tau^{-}$.

Fix $\hat x\in\widehat X$ and let $W$ be a $\widehat\tau^\vee$-neighborhood
of $\hat x$. Then there exist entourages $\widehat U_1,\widehat U_2\in
\widehat{\mathcal U}^{+}$ such that
\[
\widehat U_1(\hat x)\cap \widehat U_2^{-1}(\hat x)\subseteq W.
\]
Choose entourages $\widehat V_1,\widehat V_2\in\widehat{\mathcal U}^{+}$
with
\[
\widehat V_1\circ \widehat V_1\subseteq \widehat U_1,
\qquad
\widehat V_2\circ \widehat V_2\subseteq \widehat U_2.
\]
Since $i(X)$ is dense in $(\widehat X,\widehat\tau^\vee)$, there exists
$x\in X$ such that
\[
i(x)\in \widehat V_1(\hat x)\cap \widehat V_2^{-1}(\hat x).
\]
By local antisymmetric connectedness of $X$ at $x$, there exists a
$\tau^\vee$-neighborhood $N$ of $x$ such that $N$ is antisymmetrically
connected and small enough that
\[
i(N)\subseteq \widehat V_1(i(x))\cap \widehat V_2^{-1}(i(x)).
\]
Then for every $y\in i(N)$ we have $(\hat x,y)\in \widehat U_1$ and
$(y,\hat x)\in \widehat U_2^{-1}$ by composition, hence
\[
i(N)\subseteq \widehat U_1(\hat x)\cap \widehat U_2^{-1}(\hat x)\subseteq W.
\]
Since $i(N)$ is antisymmetrically connected and $i(N)\subseteq W$, we
conclude that $\widehat X$ is locally antisymmetrically connected at
$\hat x$.
\end{proof}

\section{Functional Analytic Consequences}
\label{sec:functional-analytic-consequences}

This section illustrates how the setting of antisymmetric and local
antisymmetric connectedness interacts with classical themes in
functional analysis, particularly in asymmetric normed and modular
spaces. The emphasis is on completion mechanisms, compactness principles,
and the behavior of linear operators when directional topology is
retained rather than suppressed by symmetrization.

\subsection{Completeness mechanisms and compactness transfers}
\label{ssec:fa-completion}

\begin{proposition}
\label{prop:fa-bicompletion}
Let $(X,\|\cdot\|)$ be an asymmetric normed space and define the
symmetrized norm
\[
\|x\|^{\vee}=\max\{\|x\|,\|{-}x\|\}.
\]
Then the bicompletion of $(X,\|\cdot\|)$ (equivalently, of the associated
quasi-uniformity $\mathcal U^{+}$) coincides, up to linear isometric
isomorphism, with the Banach space completion of the normed space
$(X,\|\cdot\|^{\vee})$.
\end{proposition}

\begin{proof}
In an asymmetric normed space, a sequence $(x_n)$ is bicomplete-Cauchy
if and only if it is Cauchy with respect to the symmetrized norm
$\|\cdot\|^{\vee}$. Indeed,
\[
\|x_n-x_m\|^{\vee}
=\max\{\|x_n-x_m\|,\|x_m-x_n\|\},
\]
so convergence in $\|\cdot\|^{\vee}$ is equivalent to simultaneous
forward and backward Cauchy control in the asymmetric norm. Consequently,
the class of bicomplete-Cauchy sequences in $X$ coincides with the class
of $\|\cdot\|^{\vee}$-Cauchy sequences.

Let $\widehat X^{\vee}$ denote the Banach space completion of
$(X,\|\cdot\|^{\vee})$, constructed as equivalence classes of
$\|\cdot\|^{\vee}$-Cauchy sequences modulo $\|\cdot\|^{\vee}$-null
sequences. The canonical embedding
$\iota:X\hookrightarrow \widehat X^{\vee}$ is linear and isometric with
respect to the symmetrized norm.

The asymmetric norm $\|\cdot\|$ extends uniquely to $\widehat X^{\vee}$
by
\[
\widehat{\|[x_n]\|}=\lim_{n\to\infty}\|x_n\|,
\]
which is well defined since
\[
\|x_n-x_m\|\le \|x_n-x_m\|^{\vee}
\]
ensures forward Cauchy control. The conjugate asymmetric norm extends
analogously, so that $\widehat X^{\vee}$ inherits a natural asymmetric
norm structure whose associated symmetric norm is precisely
$\|\cdot\|^{\vee}$.

If $(Y,\|\cdot\|_Y)$ is any bicomplete asymmetric normed space and
$T:X\to Y$ is linear and continuous with respect to the asymmetric norm,
then $T$ is bounded for the symmetrized norms and therefore extends
uniquely to a continuous linear operator
$\widehat T:\widehat X^{\vee}\to Y$. This universal property
characterizes $\widehat X^{\vee}$ as the bicompletion of $(X,\|\cdot\|)$.
\end{proof}

\medskip

\begin{theorem}
\label{thm:fa-compactness}
Let $(X,w)$ be a quasi-modular pseudometric space and let
$\mathcal U^{+}(w)$ be the forward quasi-uniformity induced by $w$
(Definition~\ref{def:quni-from-w}). Assume that $(X,\mathcal U^{+}(w))$
is Smyth complete (Definition~\ref{def:smyth-complete}). If
$(X,\mathcal U^{+}(w))$ is precompact, then the join topology
\[
\tau^\vee(w)=\tau^{+}(w)\vee\tau^{-}(w)=\tau(w^{\mathrm{sym}})
\]
is compact.
\end{theorem}

\begin{proof}
By Theorem~\ref{thm:precompact-smyth-compact}, if a $T_0$ quasi-uniform
space is precompact and Smyth complete, then its symmetrized uniform
space is compact. Applying this to $(X,\mathcal U^{+}(w))$ yields that
\[
(X,\mathcal U^{+}(w)\vee\mathcal U^{-}(w))
\]
is compact as a uniform space. By
Proposition~\ref{prop:compatibility-generation}, the topology induced by
this symmetrized uniformity is exactly $\tau(w^{\mathrm{sym}})$, and by
Proposition~\ref{prop:comparison-envelopes} we have
$\tau(w^{\mathrm{sym}})=\tau^\vee(w)$. Hence $\tau^\vee(w)$ is compact.
\end{proof}

\subsection{Linear operators and invariance of connectedness}
\label{ssec:fa-operators}

\begin{proposition}
\label{prop:fa-linear-continuity}
Let $(X,\|\cdot\|_X)$ and $(Y,\|\cdot\|_Y)$ be asymmetric normed spaces,
and let $T:X\to Y$ be linear. If $T$ is continuous with respect to the
forward topologies $\tau_X^{+}$ and $\tau_Y^{+}$, then $T$ is uniformly
continuous with respect to the induced quasi-uniformities.
\end{proposition}

\begin{proof}
The forward topology $\tau_X^{+}$ is translation invariant and generated
by forward balls
\[
B_X^{+}(0,\varepsilon)=\{x\in X:\|x\|_X<\varepsilon\},
\qquad \varepsilon>0,
\]
and similarly for $\tau_Y^{+}$. Linearity implies that continuity is
equivalent to continuity at the origin. Thus for every $\varepsilon>0$
there exists $\delta>0$ such that
\[
\|x\|_X<\delta \quad\Longrightarrow\quad \|Tx\|_Y<\varepsilon.
\]

Let $\mathcal U_X^{+}$ and $\mathcal U_Y^{+}$ denote the induced
quasi-uniformities. A basic entourage of $\mathcal U_Y^{+}$ is
\[
U_\varepsilon^Y=\{(y_1,y_2)\in Y\times Y:\|y_1-y_2\|_Y<\varepsilon\}.
\]
By linearity,
\[
(x_1,x_2)\in U_\delta^X
\quad\Longrightarrow\quad
(Tx_1,Tx_2)\in U_\varepsilon^Y,
\]
so for each entourage of $\mathcal U_Y^{+}$ there exists an entourage of
$\mathcal U_X^{+}$ whose image under $T\times T$ is contained in it.
Hence $T$ is uniformly continuous.
\end{proof}

\medskip

\begin{theorem}
\label{thm:fa-connectedness-preservation}
Let $T:X\to Y$ be a uniformly continuous linear operator between
quasi-uniform spaces. If $X$ is locally antisymmetrically connected,
then $T(X)$ is locally antisymmetrically connected in the subspace
bitopology induced from $Y$.
\end{theorem}

\begin{proof}
By Proposition~\ref{prop:uniform-images}, uniformly continuous maps
preserve local antisymmetric connectedness. Hence $T(X)$ is locally
antisymmetrically connected when equipped with the forward and backward
topologies induced by the quasi-uniformity transported along $T$. These
coincide with the subspace forward and backward topologies inherited
from $Y$, so $T(X)$ is locally antisymmetrically connected in the
subspace bitopology.
\end{proof}

\subsection{Duality, separation, and directional functionals}
\label{ssec:fa-duality}

\begin{remark}
\label{rem:fa-duality}
In asymmetric normed spaces, the dual space naturally decomposes into
forward and backward continuous linear functionals. This directional
duality reflects the intrinsic asymmetry of the norm and leads to
distinct separation and extension notion for the topologies
$\tau^{+}$ and $\tau^{-}$. See \cite{Cobzas2013,CobzasHB} for details.
\end{remark}

\medskip

\begin{proposition}
\label{prop:fa-separation}
Let $X$ be an asymmetric normed space and let $A,B\subset X$ be disjoint
$\tau^{\vee}$-closed convex sets. If $A$ and $B$ are separated by a
$\tau^{+}$-continuous linear functional, then no antisymmetrically
connected subset of $X$ meets both $A$ and $B$.
\end{proposition}

\begin{proof}
Let $f:X\to\mathbb{R}$ be a $\tau^{+}$-continuous linear functional and
$\alpha\in\mathbb{R}$ such that
\[
f(a)\le \alpha < f(b)
\qquad\text{for all } a\in A,\ b\in B.
\]
Then
\[
U=\{x\in X:\ f(x)<\alpha\}, \qquad
V=\{x\in X:\ f(x)>\alpha\}
\]
are $\tau^{+}$-open and $\tau^{-}$-open, respectively, with
$A\subseteq U$, $B\subseteq V$, and $U\cap V=\varnothing$.

If a subset $C\subseteq X$ meets both $A$ and $B$, then $C\cap U$ and
$C\cap V$ form a separation of $C$ by a forward-open and a backward-open
set. Hence $C$ cannot be antisymmetrically connected. This applies to
every subset meeting both $A$ and $B$, completing the proof.
\end{proof}

\section{Examples and Applications}\label{sec:examples-applications}

\subsection{Asymmetric normed sequence spaces}\label{ssec:ex-seq}

\begin{example}
\label{ex:one-sided-lp}
Let $1\le p<\infty$ and consider the vector space $\ell^p$ of real
sequences $x=(x_n)_{n\in\mathbb{N}}$. Define the asymmetric norm
\[
\|x\|_{p}^{+}
=
\Bigl(\sum_{n=1}^{\infty} (x_n^{+})^p\Bigr)^{1/p},
\qquad
x_n^{+}=\max\{x_n,0\}.
\]
The conjugate asymmetric norm is
\[
\|x\|_{p}^{-}=\|{-}x\|_{p}^{+}
=
\Bigl(\sum_{n=1}^{\infty} (x_n^{-})^p\Bigr)^{1/p},
\qquad
x_n^{-}=\max\{-x_n,0\},
\]
and the symmetrization satisfies
\[
\|x\|_{p}^{\vee}
=
\max\{\|x\|_{p}^{+},\|x\|_{p}^{-}\}
=
\|x\|_{p}.
\]

The induced forward and backward topologies coincide with the classical
upper and lower $\ell^p$ topologies studied in asymmetric functional
analysis. In particular, the associated join topology agrees with the
usual norm topology induced by $\|\cdot\|_{p}^{\vee}$, illustrating the
general symmetrization principle described in
Section~\ref{ssec:an} and Proposition~\ref{prop:fa-bicompletion}.
Moreover, local antisymmetric connectedness is understood relative to
the bitopology $(\tau^{+},\tau^{-})$ in the sense of
Definitions~\ref{def:antisym-connected} and \ref{def:locally-antisym}.
This example demonstrates how antisymmetric connectedness detects
genuinely one-sided features that disappear at the symmetric level; see
\cite{Cobzas2013}.
\end{example}


\subsection{Directed graphs and cost models}\label{ssec:ex-graphs}

\begin{example}
\label{ex:lawvere-graph}
Let $G=(V,E,w)$ be a directed graph with vertex set $V$, edge set
$E\subseteq V\times V$, and weight function
$w:E\to[0,\infty]$. Define
\[
d(x,y)
=
\inf\Bigl\{\sum_{i=0}^{n-1} w(v_i,v_{i+1})
:\ x=v_0\to v_1\to\cdots\to v_n=y\Bigr\},
\]
with $d(x,y)=\infty$ if no directed path exists from $x$ to $y$.

The function $d$ is a Lawvere quasi-metric, inducing forward and backward
topologies as in Section~\ref{ssec:qm}. The associated join topology
$\tau^\vee=\tau_d\vee\tau_{d^{-1}}$ captures the symmetrized envelope
(Definition~\ref{def:tau-vee-connected}), while antisymmetric
connectedness (Definition~\ref{def:antisym-connected}) encodes the
absence of separations by a forward-open set and a backward-open set.
In this setting, antisymmetric connectedness typically differs from
connectedness in the symmetrized metric unless the graph is strongly
connected. This model underlies applications in optimization, transport,
and computational semantics; see \cite{Lawvere1973}.
\end{example}


\subsection{Orlicz-type modular settings}\label{ssec:ex-orlicz}

\begin{example}
\label{ex:orlicz}
Let $(\Omega,\Sigma,\mu)$ be a measure space and let
$\Phi:\Omega\times\mathbb{R}\to[0,\infty]$ be a Musielak-Orlicz function
that is convex and increasing in the second variable, but not
necessarily even. Define the modular
\[
\rho(f)
=
\int_{\Omega}\Phi(\omega,f(\omega))\,d\mu(\omega),
\]
and set
\[
w_{\lambda}(f,g)
=
\rho\!\left(\frac{g-f}{\lambda}\right),
\qquad \lambda>0.
\]

The family $w=(w_\lambda)_{\lambda>0}$ defines a quasi-modular
pseudometric structure in the sense of
Definition~\ref{def:quasi-modular-family}, with conjugation and
symmetrization as in Definition~\ref{def:conjugate-symmetrized}. It
generates forward and backward modular topologies $\tau^{+}(w)$ and
$\tau^{-}(w)$ (Definition~\ref{def:forward-backward-topologies}) and
compatible quasi-uniformities $\mathcal U^{+}(w)$ and $\mathcal U^{-}(w)$
(Definition~\ref{def:quni-from-w} and
Proposition~\ref{prop:compatibility-generation}). Local antisymmetric
connectedness reflects one-sided control of modular displacement in the
sense of Definition~\ref{def:locally-antisym}, and is generally not
preserved under symmetrization unless $\Phi(\omega,t)=\Phi(\omega,-t)$.
This example connects the present theory with modular analysis and
nonlinear function spaces; see
\cite{Chistyakov2015,Majozi2025,Musielak1983}.
\end{example}


\subsection{Paratopological groups and semigroups}\label{ssec:ex-algebra}

\begin{remark}
\label{rem:ex-paratopo}
Paratopological groups and semigroups provide natural algebraic settings
in which left and right uniformities differ. The associated bitopological
structure captures algebraic asymmetry, while antisymmetric
connectedness distinguishes algebraic reachability from classical
topological connectedness (Definitions~\ref{def:antisym-connected} and
\ref{def:tau-vee-connected}). Completion and compactness notions in
this context are closely linked to Do\u{\i}tchinov completeness, Smyth
completeness, stability, and bicompletion theory in the quasi-uniform
setting; see
\cite{Kunzi2001,KunziJunnila1993,KunziRomagueraSipacheva1998}.
\end{remark}

\section{Conclusion and Further Directions}\label{sec:conclusion}

This paper develops a setting for connectedness in genuinely
nonsymmetric environments. Within quasi-uniform and quasi-modular
pseudometric spaces we introduced \emph{antisymmetric connectedness} and
\emph{local antisymmetric connectedness}, capturing directional cohesion
that is invisible after symmetrization to the join topology
$\tau^{\vee}=\tau^{+}\vee\tau^{-}$. We related these notions to symmetric
connectedness, compared the corresponding component theories, and gave
examples showing that antisymmetric components may be strictly finer
than the classical ones. We also obtained practical local
characterizations in terms of entourages and modular balls, and proved
stability under $\tau^{\vee}$-open subspaces, finite unions with
nonempty intersection, uniformly continuous images, and bicompletion.

Several questions remain open. It is natural to study how antisymmetric
connectedness behaves under other completion procedures (Do\u{\i}tchinov
and Smyth completions), to seek quantitative refinements in the setting
of approach spaces, and to develop categorical and domain-theoretic
interpretations via enriched methods and formal balls. From the analytic
side, further work may clarify the role of local antisymmetric
connectedness in fixed point theory, variational principles, and
nonlinear analysis on asymmetric normed and modular spaces.


\begin{thebibliography}{999}

\bibitem{Chistyakov2010}
V.~V.~Chistyakov,
\emph{Modular metric spaces I: Basic concepts},
Nonlinear Anal. \textbf{72} (2010), 1--14.

\bibitem{Chistyakov2015}
V.~V.~Chistyakov,
\emph{Metric Modular Spaces},
Springer, Cham, 2015.

\bibitem{Cobzas2006}
S.~Cobza\c{s},
Compact operators on spaces with asymmetric norm,
\emph{Acta Math. Hungar.} \textbf{110} (2006), 219--239.

\bibitem{Cobzas2013}
S.~Cobza\c{s},
\emph{Functional Analysis in Asymmetric Normed Spaces},
Birkh\"auser/Springer, Basel, 2013.

\bibitem{CobzasHB}
S.~Cobza\c{s},
Hahn--Banach type extension results for linear operators on asymmetric normed spaces,
\emph{arXiv preprint} arXiv:2412.10391, 2024.

\bibitem{CobzasSurvey}
S.~Cobza\c{s},
Completeness in quasi-pseudometric spaces -- A survey,
\emph{Mathematics} \textbf{8} (2020), Article 1279.

\bibitem{Doitchinov1991}
D.~Doitchinov,
A concept of completeness of quasi-uniform spaces,
\emph{Topology Appl.} \textbf{38} (1991), 205--217.

\bibitem{FlaggSuenderhauf2002}
R.~C.~Flagg and P.~S\"underhauf,
\emph{The essence of ideal completion in quantitative form},
Theoretical Computer Science \textbf{278} (2002), no.~1--2, 141--158.

\bibitem{FlaggSuenderhaufWagner1997}
R.~C.~Flagg, P.~S{\"u}nderhauf and K.~R.~Wagner,
\emph{A Logical Approach to Quantitative Domain Theory},
Topology Atlas Preprint No.~23 (1996).

\bibitem{FletcherLindgren}
P.~Fletcher and W.~F.~Lindgren,
\emph{Quasi-Uniform Spaces},
Marcel Dekker, New York, 1982.

\bibitem{ForalewskiHudzikKolwicz}
P.~Foralewski, H.~Hudzik, and P.~Kolwicz,
Quasi-modular spaces with applications to quasi-normed
Calder\'on--Lozanovski\u{\i} spaces,
\emph{arXiv preprint} arXiv:2208.08970, 2022.

\bibitem{Isbell1964}
J.~R.~Isbell,
\emph{Uniform Spaces},
Mathematical Surveys, No.~12, American Mathematical Society, Providence, RI, 1964.

\bibitem{JavanshirYildiz2023}
N.~Javanshir and F.~Yıldız,
\emph{Local antisymmetric connectedness in asymmetrically normed real vector spaces},
Universal Journal of Mathematics and Applications \textbf{6} (2023), no.~3, 100--105.

\bibitem{Kelley}
J.~L.~Kelley,
\emph{General Topology},
Springer, New York, 1955.

\bibitem{Kelly1982}
G.~M.~Kelly,
\emph{Basic Concepts of Enriched Category Theory},
Cambridge University Press, Cambridge, 1982.

\bibitem{Kunzi2001}
H.-P.~K\"unzi,
Nonsymmetric distances and their associated topologies,
in: \emph{Handbook of the History of General Topology}, Vol.~3,
Kluwer Academic Publishers, 2001, pp.~853--968.

\bibitem{KunziJunnila1993}
H.-P.~A.~K\"unzi and H.~J.~K.~Junnila,
Stability in quasi-uniform spaces and the inverse problem,
\emph{Topology Appl.} \textbf{49} (1993), 175--189.

\bibitem{KunziLuethy1995}
H.-P.~A.~K\"unzi and A.~L{\"u}thy,
Dense subspaces of quasi-uniform spaces,
\emph{Studia Sci. Math. Hung.} \textbf{30} (1995), 289--301.

\bibitem{KunziRomaguera}
H.-P.~K\"unzi and S.~Romaguera,
Weightable quasi-uniformities and quasi-pseudometrics,
\emph{Topology Appl.} \textbf{79} (1997), 257--273.

\bibitem{KunziRomagueraSipacheva1998}
H.-P.~A.~K\"unzi, S.~Romaguera and O.~V.~Sipacheva,
The Doitchinov completion of a regular paratopological group,
\emph{Serdica Math. J.} \textbf{24} (1998), 73--88.

\bibitem{KunziSchellekens2002}
H.-P.~A.~K\"unzi and M.~P.~Schellekens,
\emph{On the Yoneda completion of a quasi-metric space},
Theoret.\ Comput.\ Sci. \textbf{278} (2002), no.~1, 159--194.

\bibitem{Lawvere1973}
F.~W.~Lawvere,
Metric spaces, generalized logic, and closed categories,
\emph{Rend. Sem. Mat. Fis. Milano} \textbf{43} (1973), 135--166.

\bibitem{Lowen1989}
R.~Lowen,
Approach spaces: a common supercategory of \textsc{Top} and \textsc{Met},
\emph{Math.\ Nachr.} \textbf{141} (1989), 183--226.

\bibitem{Lowen1997}
R.~Lowen,
\emph{Approach Spaces: The Missing Link in the Topology-Uniformity-Metric Triad},
Oxford Mathematical Monographs, Oxford University Press, 1997.

\bibitem{Lowen2000}
R.~Lowen and B.~Windels,
Quantifying completion,
\emph{Internat.\ J.\ Math.\ Math.\ Sci.} \textbf{23} (2000), no.~11, 729--739.

\bibitem{Majozi2025}
P.~R.~Majozi,
On quasi-modular pseudometric spaces and asymmetric uniformities,
\emph{arXiv preprint} arXiv:2510.25200, 2025.

\bibitem{MarinRomaguera1998}
J.~Mar\'{\i}n and S.~Romaguera,
Bicompleting the left quasi-uniformity of a paratopological group,
\emph{Arch. Math. (Basel)} \textbf{70} (1998), 104--110.

\bibitem{MushaandjaOlela2025}
Z.~Mushaandja and O.~Olela-Otafudu,
\emph{On the modular metric topology},
Topology Appl. \textbf{372} (2025), 109224.

\bibitem{Musielak1983}
J.~Musielak,
\emph{Orlicz Spaces and Modular Spaces},
Lecture Notes in Mathematics, Vol.~1034,
Springer, Berlin, 1983.

\bibitem{RomagueraTiradoValero2012}
S.~Romaguera, P.~Tirado and O.~Valero,
Complete partial metric spaces have partially metrizable computational models,
\emph{Internat.\ J.\ Comput.\ Math.} \textbf{89} (2012), no.~3, 284--290.

\bibitem{RomagueraValero2009}
S.~Romaguera and O.~Valero,
A quantitative computational model for complete partial metric spaces via formal balls,
\emph{Math. Structures Comput. Sci.} \textbf{19} (2009), no.~3, 541--563.

\bibitem{Rutten1998}
J.~J.~M.~Rutten,
Weighted colimits and formal balls in generalized metric spaces,
\emph{Topology Appl.} \textbf{89} (1998), 179--202.

\bibitem{Stoltenberg1967}
R.~A.~Stoltenberg,
Some properties of quasi-uniform spaces,
\emph{Proc. London Math. Soc.} \textbf{17} (1967), 226--240.

\bibitem{Sunderhauf1997}
P.~S{\"u}nderhauf,
Tensor products and powerspaces in quantitative domain theory,
\emph{Electron. Notes Theor. Comput. Sci.} \textbf{6} (1997), 327--347.

\bibitem{Sunderhauf1998}
P.~S{\"u}nderhauf,
Spaces of valuations as quasimetric domains,
\emph{Electron. Notes Theor. Comput. Sci.} \textbf{13} (1998).

\bibitem{Vickers1997}
S.~J.~Vickers,
Localic completion of quasimetric spaces,
Technical Report DoC 97/2, Department of Computing,
Imperial College, London, 1997.

\bibitem{Willard}
S.~Willard,
\emph{General Topology},
Addison--Wesley, Reading, MA, 1970.

\end{thebibliography}
\end{document}